\documentclass[reqno]{amsart}
\usepackage{amsmath,amssymb,amsthm}
\usepackage{amscd}
\usepackage[mathscr]{eucal}


\usepackage{graphicx}
\usepackage{amssymb}
\usepackage{amsmath}
\usepackage{color}
\usepackage{times}

\usepackage[unicode,bookmarks,colorlinks]{hyperref}
\hypersetup{
    linkcolor=brickred,
}

\definecolor{brickred}{cmyk}{0, 0.89, 0.94, 0.28}

\begin{document}
\newtheorem{thm}{Theorem}
\numberwithin{thm}{section}
\newtheorem{lemma}[thm]{Lemma}
\newtheorem{remark}{Remark}
\newtheorem{corr}[thm]{Corollary}
\newtheorem{proposition}{Proposition}
\newtheorem{Problem}{Problem}
\newtheorem{example}{Example}
\newtheorem{theorem}{Theorem}[section]
\newtheorem{deff}[thm]{Definition}
\newtheorem{note}{Note}
\newtheorem{question}{Question}
\newtheorem{case}[thm]{Case}
\newtheorem{prop}[thm]{Proposition}
\numberwithin{equation}{section}
\numberwithin{remark}{section}
\numberwithin{proposition}{section}
\newcommand{\uG}{\underline{G}}
\newcommand{\bD}{\mathrm{I\! D\!}}
\newcommand{\bR}{\mathrm{I\! R\!}}
\newcommand{\uT}{\underline{T}}
\newcommand{\uB}{\underline{B}}
\newcommand{\uU}{\underline{U}}
\newcommand{\bH}{\mathrm{I\! H\!}}
\newcommand{\uA}{\underline{A}}
\newcommand{\uM}{\underline{M}}
\newcommand{\uN}{\underline{N}}
\newcommand{\uP}{\underline{P}}
\newcommand{\sN}{\cal N}
\newcommand{\bC}{\Bbb C}
\newtheorem{corollary}{Corollary}[section]
\newtheorem{others}{Theorem}
\newcommand{\nucleo}{p_{D}_{t}(x,y)}
\newtheorem{conjecture}{Conjecture}
\newtheorem{definition}{Definition}[section]
\newtheorem{cl}{Claim}
\newtheorem{cor}{Corollary}
\newcommand{\ds}{\displaystyle}
\date{}
\newcommand{\pa}{{\cal P}_{\alpha}}

\title{Dimension for martingales}
\author{Prabhu Janakiraman}\address{Prabhu Janakiraman, 
{\tt pjanakir1978@gmail.com}}

\maketitle
\begin{abstract}
We define that a martingale $\int H\cdot dZ$ has Dimension $k$ if the rank of the matrix process $H$ equals $k$ almost surely, for almost every t. The definition is shown to be well defined, and the value can be used as a geometric invariant to classify and study martingales. We also define  general Brownian motions in higher dimensions. 
\end{abstract}

\section{Introduction}
This note is an addendum to the ideas introduced in section $3$ of \cite{Ja1}. Our purpose is to introduce the notion of dimension for a martingale. Using this concept, it seems that one can deal with martingales in the way topologists deal with manifolds. Some of the concepts like submersions and related theorems, as found in \cite{GP}, seem accessible. We give a starting point for such a theory and prove some preliminary theorems. (Martingale-dimension is not a new topic, however; see section \ref{AF_dim} for a brief discussion on this. The concept introduced in this paper appears to be new.)

In section 3.3 of \cite{Ja1}, the author defined a $\bC^1$-Brownian motion $X= X_1+iX_2$, as opposed to regular complex or $\bR^2$ Brownian motion, as a continuous martingale with $X_0=0$ and satisfying
$\langle X\rangle_t = \langle X_1\rangle_t + \langle X_2\rangle_t \equiv t$.
This we called the L\'evy characterization for $\bC^1$-Brownian motion. One of the interesting consequences of this definition is that if $(Z_1, Z_2)$ is an $\bR^2$-Brownian motion, then $Z_1$ and $\frac{Z_1 + i Z_2}{\sqrt{2}}$ become equivalent entities when considered as $\bC^1$-processes.
 At the end of the same section, the author raised the question of finding a basic property that characterizes (general) Brownian motion independently of the space in which it travels. Such a property is then an \emph{intrinsic invariant} of Brownian motion.

\begin{definition}
A common property or value characterizing a collection of processes is an invariant of that collection. 
\end{definition}
We may call the property \emph{intrinsic} if it is dependant fundamentally on the structures or variables that define the processes; and in general, not defined directly in terms of the range space of the processes. The L\'evy characterization already offers the quadratic variation as a basic invariant. More generally, we can say that any two martingales on a probability space $\Omega$ are equivalent if they have the same quadratic variation (QV). In particular, we may consider identifying all processes with QV $= t$ as Brownian motion, seek to find their shared properties and transforms that preserve them.

\begin{note} 
\begin{enumerate}
\item
All martingales that we consider are measuarable with respect to some Brownian filtration, having a stochastic representation $\int H\cdot dZ$, such that ${\bf{P}}(\int_0^t |H_s|^2 ds < \infty) =1$ for all $t$. 
\item For any matrix $A$, $A^{tr}$ means the transpose of $A$.
\end{enumerate}
\end{note}

\section{Dimension}
On the other hand, it is natural to think of $(Z_1, 0)$ and $\frac{(Z_1, Z_2)}{\sqrt{2}}$ as distinct when considered as $\bR^2$ processes. As is obvious, $(Z_1, 0)$ maps into a $1$-dimensional subspace of $\bR^2$ whereas $\frac{(Z_1, Z_2)}{\sqrt{2}}$ has a genuinely $2$-dimensional range space. But such a property requires that we think of the process in terms of its range space; it is extrinsic.  We want to establish the distinction of such processes in terms of an `intrinsic' invariant for martingales.

\begin{definition}\label{Dim_def}
Let $X=\int H\cdot dZ$ be an $\bR^n (\bC^n)$ valued martingale that is  measurable with respect to $\bR^d$-Brownian motion $Z$. We say that $X$ has real (complex) Dimension $k$, write rDim (cDim), if the real (complex) matrix process $H_t$ has rank $k$ almost surely  for almost every $t>0$. 
\end{definition}

We will establish that Dimension is a well-defined property that does not depend on the particular stochastic representation. Observe that $rDim(Z_1)=1$ and $rDim(Z_1+iZ_2)= rDim(Z_1,Z_2) =2$, whereas $cDim(Z_1) = cDim(Z_1+iZ_2)=1$. 
\begin{remark}
Here on, we will deal only with $\bR^n$-valued martingales, and refer to $rDim$ as simply Dimension. The concepts can also be extended to the $\bC^n$ setting. 
\end{remark}
 Although a matrix indicates the range space's dimension, we think of the rank as an intrinsic property of the linear transformation, indicating the minimal number of linearly independent rows or columns rather than the total size. Likewise, the dimension tells us that the martingale is intrinsically or locally equivalent to a Brownian motion of that unique dimension.
In giving Definition \ref{Dim_def}, the author has the following analogy in mind. When we consider a manifold $M$ of dimension $d$, we think that at any point $p\in M$ there exists a neighborhood that is homeomorphic to a coordinate neighborhood of $\bR^d$. Likewise, a martingale $\int H\cdot dZ$ of dimension $k$ should be ``locally'' equivalent to $\bR^k$-Brownian motion $W$ via a martingale transform:
$$ W \rightarrow K\star W = \int K\cdot dW = H\star Z = \int H\cdot dZ.$$
Heuristically, $H$ takes the $d$-dimensional $dZ$ onto a $k$-dimensional $H\cdot dZ$ in $\bR^n$. And the base process $Z$ splits into an orthogonal sum of $S + S^\perp$ where $S$ and $S^\perp$ are $k$ and $d-k$ dimensional processes respectively. This idea is explored more precisely in the section 4.

\begin{remark}
More generally, given a martingale $X=(X^1, \cdots, X^n)$, we can consider the covariation matrix process $B = [b_{ij}]_{i,j=1}^n$ where 
$$ b_{ij} = \frac{d}{dt}\langle X^i, X^j\rangle_t .$$ We can then say $Dim(X) = k$ if $rank(B)=k$ a.s., for a.e. $t$. This agrees with the earlier definition since if $X=\int H\cdot dZ$, then $B = HH^{tr}$ and $rank(H) = rank(B)$. This definition also indicates the intrinsic nature of Dimension, since it is directly based on the covariance relations of the coordinates. We also see that, besides quadratic variation and dimension, the covariance matrix process also gives equivalence classes of $\bR^n$ valued martingales.
\end{remark}
As an added comment to this remark, it must be admitted that Dimension, unlike the quadratic variation, cannot be completely independent of the range space. A process of dimension $k$ travels in some space of dimension greater than or equal to $k$. We can however overcome the difficulty by identifying all processes as taking values in $\bR^\infty$; doing this should also broaden the scope for this type of analysis on martingales. 

\subsection{The key questions}
While we have proposed Definition \ref{Dim_def} for dimension, it is not clear whether it is well-defined. We have to resolve the following two questions. 

\begin{question}\label{question1}
\begin{enumerate}
\item Is Dimension a well-defined property for a martingale that is independent of the stochastic representation?
\item If $Dim(X) = k$, is there an $\bR^k$-Brownian motion $W$ such that $X=\int K\cdot dW$?
\end{enumerate}
\end{question}
Before we proceed to answer, let us consider an alternate definition for dimension.

\begin{definition}\label{Brownian_D}
An $\bR^n$-valued martingale $X=\int H\cdot dZ$ has Brownian Dimension $k$ if
\begin{enumerate}
\item
$X$ is measurable with respect to a $k$-dimensional Brownian motion $W$ and has the stochastic representation $\int K\cdot dW$,
\item The matrix process $K$ has rank $k$ almost surely for almost every $t>0$.
\end{enumerate}
\end{definition}
Observe that this is apparently a stronger requirement than Def \ref{Dim_def}. The implication of Question $(2)$ is assumed apriori, and the corresponding matrix $K$ has always rank $k$. Therefore, this is a well-defined notion for dimension. But the subtle difference between the two definitions is explained in the next section. They don't appear to be equivalent even though both questions in Question \ref{question1} have affirmative answers.

\begin{example}
Let $X=\int H\cdot dZ$ be $\bR^n$ valued and measurable with respect to $d$ dimensional $Z$. Define the Graph of $X$ as $Graph(X) = (Z, X)$, an $\bR^{n+d}$ martingale again run on $Z$. It is an easy to see that $Graph(X)$ is measurable with respect to $Z$ and has Brownian dimension $d$.
\end{example}

\section{Basic Results}\label{Results}
Let us begin by showing that  dimension is a well-defined property.
\begin{theorem}\label{mainT}
If $X = \int H\cdot dZ = \int N\cdot dM$ satisfies that $rank(H)= k$ a.s., a.e. t, then $rank(N)=k$  a.s., a.e. t.
\end{theorem}
\begin{proof}
Consider first the special case when $n=1$. Then $H$ and $N$ are $1\times d$ and $1\times m$ matrix processes. Since $X=\int H\cdot dZ = \int N\cdot dM$ for all t a.s., we know that $\langle X\rangle = \int |H|^2 ds = \int |K|^2 ds$ for all t, a.s. This means that for almost every $\omega$, we have $|H_s(\omega)|^2 = |K_s(\omega)|^2$ for almost every $s$. This means $rank(N) = rank(H) =1$ a.s. for almost every $t$.

Now consider the general case for any $n$. $H$ and $N$ are $n\times d$ and $n\times m$ matrix processes respectively. Following the same argument, we conclude $\langle X_i, X_j\rangle = \int H^i\cdot H^j ds = \int N^i\cdot N^j ds$ for all t, a.s. This means that for almost every $\omega$, we have $H^i_s(\omega)\cdot H^j_s(\omega) = N^i_s(\omega)\cdot N^j_s(\omega)$ for almost every $t$. It is clear that there is a rank preserving transformation between the subspaces spanned by these vectors. We conclude $rank(N) = rank(H) =k$ a.s. for almost every $t$.

\end{proof}

Next let us proceed to Question (2) and show that a $k$-dimensional martingale is equivalent to an $\bR^k$-Brownian motion.
\begin{theorem}\label{mainT2}
Let $X$ be an $\bR^n$-martingale measurable with respect to $d$-dimensional Brownian motion $Z$. Let $k\leq d\wedge n$. If $X=\int H\cdot dZ$ satisfies that $rank(H) = k$ a.s. for almost every $t$, then there exists a $\bR^k$-Brownian motion $W$ and an $n\times k$ matrix process $K$ of rank $k$ (and measurable with respect to $Z$) such that $X= \int K\cdot dW$.
\end{theorem}
\begin{proof}
Suppose $H$ is $n\times d$ matrix process of rank $k$. Take the first $k$ row vectors $\vec{v}_1, \cdots, \vec{v}_k$ that are linearly independent and let $\vec{u}_1, \cdots, \vec{u}_k$ be the orthonormal basis obtained via the Gram-Schmidt process. Let $V$ be the $k\times d$ matrix with rows $\vec{u}_j$. We can rewrite $H = K\cdot V$ for $K$, a predictible $n\times k$ matrix process.

Define $dW = V\cdot dZ$. $W$ is a $k$-dimensional continuous process starting at $0$. Moreover, because of the orthonormality of the vectors $\vec{u}_j$, we know that $\langle W_i, W_j\rangle_t = \delta_{ij} t$. By the L\'evy characterization, we conclude $W$ is an $\bR^k$-Brownian motion. We have shown $X = K\star W$ as required. That $rank(K)=k$ a.s., a.e. $t$, follows from Theorem \ref{mainT}.
\end{proof}

The theorem is intuitively what we want except for the discordant possibility that $K$ is measurable with respect to $Z$ and not necessarily with respect to $W$. One cannot casually bypass this possibility, since for example, we see that a martingale such as $\int F(Z_1, Z_2) dZ_1$ has an integrand that clearly is measurable with respect to the joint process $(Z_1, Z_2)$. Can it (or when can it) be rewritten as a stochastic integral $\int k dW$, where $W$ is $\bR^1$ Brownian motion and $k$ is predictible with respect to the filtration of $W$? The problem is even more perplexing in higher dimensions.

I. Karatzas and S. Shreve record the following theorem and remark; see Theorem 4.2, remark 4.3 in \cite{KS}, which essentially prove our results in a much more general sense, and also give a partial answer to our perplexing problem. We write it in a slightly modified manner suitable to our assumptions.
\begin{theorem}\label{KS_theorem}
Suppose $X=\{\int H\cdot dZ, \mathcal{F}\}$ is defined on $(\Omega, \mathcal{F}, P)$, where $H$ is an $n\times d$ predictible matrix process having rank $k$ a.s., a.e. t. If $H^i$ is the $i^{th}$ row, suppose $\int_0^t H^i\cdot H^j ds$ is almost surely an absolutely continuous function of $t$. Then there exists a $\bR^k$ Brownian motion $W$ on $(\Omega, \mathcal{F}, P)$ and an $n\times k$ predictible matrix process $K$ (of rank $k$) such that we have the stochastic representation
$$ X = \int K\cdot dW.$$
Moreover, $K = H\cdot V^{tr}$ and $dW = V\cdot dZ$ for a $k\times d$ predictible matrix process $V$ with orthonormal rows.
\end{theorem}
The proof of the theorem and the remark 4.3 investigate the matrix process $HH^{tr}$, which is assumed as having constant rank $k$. Since $rank(H) = rank(H H^{tr})$, our assumption is equivalent to the requirement of this proof. That $K=H\cdot V^{tr}$ and $dW= V\cdot dZ$ follow from a careful analysis of the proof.

\subsection{How is $X\sim W$?}
Notice that the theorem states that the $W$ is measurable with respect to the filtration generated by $X$, which can be smaller than that of $Z$. This is already a deeper insight than what Theorem \ref{mainT2} asserts. However, is it true the other way around ... is $X$ measurable with respect to the filtration generated by $W$? As far as the author can tell, the proof in \cite{KS} does not imply this, and it is likely not the case. If we follow their proof for the martingale $X = \int Z_2 dZ_1$, we get $K = [|Z_2|]$ and $dW = \textrm{sgn} (Z_2) dZ_1$; and $|Z_2|$ is not measurable with respect to $\int \textrm{sgn}(Z_2)dZ_1$. We leave the issue as a question with a likely negative answer.

Thus, a martingale $(X, \mathcal{F})$ of dimension $k$ can be written as $\int K\cdot dW$, where $W$ is $\bR^k$ Brownian motion,  measurable with respect to the same filtration $\mathcal{F}$, and $rank(K)=k$ a.s., for a.e. $t$. It is in this sense that we can say $X$ and $W$ are equivalent, or that $X$ runs on $W$. The author's intuition had suggested that a $k$ dimensional $X$ should be measurable with respect to the filtration of a $k$ dimensional Brownian motion. This no longer seems correct; the issue appears similar to how a function on $\bR^2$ need not be measurable with respect to the Borel sigma algebra.

Alternately, we could take Brownian Dimension of Definition \ref{Brownian_D} as our essential value. This would bypass our concern at the start. The question then becomes whether we can build a standard geometry-based theory based on Brownian dimension? We would want this theory, if it develops, to investigate martingale transforms by matrix processes that will send martingale $X$ of Dimension $k$ to martingale $Y$ of Dimension $r$. If we replace with $BDim$ (Brownian Dimension), then the class of valid martingales and transformations should become considerably smaller.

\begin{question}\label{indep_orth}
\begin{enumerate}
\item Is there a minimal set of conditions that if satisfied will imply that two orthogonal martingales are independent?  (Recall $X$ and $Y$ are orthogonal if $XY$ is a martingale.)
\item $K_s$ and ``$dW_s$" are orthogonal random variables, since $K^{ij}_sdW^r_s$ always satisfies the martingale condition. Is it always possible to choose the minimal representation so that they are independent?
\end{enumerate}
\end{question}


\section{Definition for $\bR^{n\times m}_K$ Brownian motion}
The matrix process $V$ of Theorem \ref{mainT2} is $k\times d$ with orthonormal rows $\vec{u}_j$, $j=1, \cdots, k$. It can be extended to an orthonormal basis-process of $\bR^d$, by tagging on $\vec{u}_{k+1}, \cdots, \vec{u}_d$. Let $U_i = \int \vec{u}_i\cdot dZ$, $i=1,\cdots, d$ be $d$ processes in $\bR^d$. It is evident by construction that $U=(U_1, \cdots, U_d)$ is an $\bR^d$-Brownian motion. Now consider the $\bR^d$-valued processes
\begin{eqnarray}
\vec{U}^i &=& (\vec{U}^i_1, \cdots, \vec{U}^i_d)^{tr} \nonumber \\
&=& \int \vec{u}_i \vec{u}_i\cdot dZ = (\int u_{i,1} \vec{u}_i\cdot dZ, \cdots, \int u_{i, d}\vec{u}_i\cdot dZ)^{tr}. \label{U}
\end{eqnarray}

The mutual covariation of the coordinate processes $\vec{U}^i_k$ and $\vec{U}^j_r$ is 

\begin{eqnarray*}
\langle \vec{U}^i_k, \vec{U}^j_r\rangle &=& \int u_{i,k}u_{j,r}d\langle \vec{u}_i\star Z, \vec{u}_j\star Z\rangle \\
&=& \int u_{i,k}u_{j,r}\delta_{ij} dt.
\end{eqnarray*}
Therefore, we conclude
\begin{enumerate}
\item $\langle \vec{U}^i_k, \vec{U}^j_r\rangle = 0$, for all $k, r$, whenever $i \neq j$,
\item $\langle U^i\rangle = \sum_{k=1}^d \langle U^i_k\rangle = t$, for all $i$.
\end{enumerate}
We recall the definition of orthogonality between two $\bR^n$ martingales; see \cite{BW}. 

\begin{definition}\label{orth_def}
$X= (X_1, \cdots, X_n)$ and $Y=(Y_1, \cdots, Y_n)$ are mutually orthogonal if for all $i, j$, we have $\langle X_i, Y_j\rangle = 0$, a.s. 
\end{definition}
It follows that the vector processes $\vec{U}^i$ and $\vec{U}^j$ are mutually orthogonal for all $i\neq j$, and each has total quadratic variation equal to $t$. We define a general multi-dimensional Brownian motion that is characterized by these properties.

\begin{definition}\label{BM}
Let $\vec{W}^1, \cdots, \vec{W}^m$ be a collection of $\bR^n$-valued continuous processes starting at $0$, that satisfy:
\begin{enumerate}
\item $|\vec{W}^j|^2 - t$ is a martingale for all $i$,
\item $\vec{W}^i$ and $\vec{W}^j$ are mutually orthogonal for all $i\neq j$,
\item $Dim(\vec{W}^j)=k_j$ for all $j$, 
\end{enumerate}
Let $K= (k_1, \cdots, k_m)$. Then we say that $W= (\vec{W}^1, \cdots, \vec{W}^m)$ is an $\bR^{n\times m}_{K}$ Brownian motion. 
If $m=1$ and $K=k$, then we may also refer to $W$ as an $\bR^{n}_{k}$  Brownian motion.
\end{definition}

The standard $\bR^d$ Brownian motion $Z=(Z_1, \cdots, Z_d)$ is an $\bR^{1\times d}_{{\bf{1}}_d}$ Brownian motion; ${{\bf{1}}_d}$ is defined below. A simple corollary to the definition is the following.
\begin{corollary}
A $k$ dimensional martingale $X$ in $\bR^n$ is a time change of an $\bR^{n}_{k}$ Brownian motion.
\end{corollary}

 Next, let us denote

\begin{equation}
{\bf{1}}_k = (1, \cdots, 1) (\textrm{ k times}).
\end{equation}
Each $\vec{U}^j$ of (\ref{U}) is an $\bR^{d}_{1}$ Brownian motion, and the processes $\vec{U} = (\vec{U}^1, \cdots, \vec{U}^d)$, $(\vec{U}^1, \cdots, \vec{U}^k)$ and $(\vec{U}^{k+1}, \cdots, \vec{U}^d)$ are $\bR^{d\times d}_{{\bf{1}}_d}$, $\bR^{d\times k}_{{\bf{1}}_k}$ and $\bR^{d\times {d-k}}_{{\bf{1}}_{d-k}}$ Brownian motions respectively. The processes 
$$ S^1 = \frac{\vec{U}^1+\cdots + \vec{U}^k}{\sqrt{k}} \textrm{   and     } S^2 = \frac{\vec{U}^{k+1}+\cdots + \vec{U}^{d}}{\sqrt{d-k}}$$ are $\bR^{d}_k$ and $\bR^{d}_{d-k}$-Brownian motions, that are mutually orthogonal to one another. Therefore, $(S_1, S_2)$ is an $\bR^{d\times 2}_{(k, d-k)}$ Brownian motion. And their normalized sum $\frac{S^1+S^2}{\sqrt{2}}$ is an $\bR^d_d$ Brownian motion. Plus, observe $\sqrt{k}S^1+\sqrt{d-k}S^2 = Z^{tr}$. (These assertions are easy consequences of Proposition \ref{special_property}).

\begin{proposition}
The choice of the basis $\vec{u}_1, \cdots, \vec{u}_k$ would not affect the sum process $\sqrt{k} S^1 = \vec{U}_1 + \cdots + \vec{U}_2$. 
\end{proposition}
\begin{proof}
Take another orthonormal collection  $\vec{v}_1, \cdots, \vec{v}_k$ that spans the same $k$ dimensional subspace of $\bR^d$, and 
define $\vec{V}_j$ in exactly the same manner as we defined $\vec{U}_j$. Let $\sqrt{k}T^1 = \vec{V}_1 +\cdots + \vec{V}_k$. We wish to show that for any $n\times d$ predictible matrix process $J$, we have $$J\star T^1 = J\star S^1.$$ Define the $d\times k$ matrices
\[
A= \begin{pmatrix} \vec{u}_1 & \cdots & \vec{u}_k\end{pmatrix}, B= \begin{pmatrix} \vec{v}_1 & \cdots & \vec{v}_k\end{pmatrix}
\] 
Then
\begin{eqnarray*}
S^1 &=& \int \sum_{j=1}^k \vec{u}_j \vec{u}_j\cdot dZ = \int [A\cdot A^{tr}]\cdot dZ \\
T^1 &=& \int \sum_{j=1}^k \vec{v}_j \vec{v}_j\cdot dZ = \int [B\cdot B^{tr}]\cdot dZ.
\end{eqnarray*}
Now extend the collections to bases of all of $\bR^d$: $\tilde{A} = [A\hspace{3mm} A_{*}]$ and $\tilde{B} = [B\hspace{3mm} B_{*}]$ where $A_* = (u_{k+1} \cdots u_d)$ and $B_* = (v_{k+1} \cdots v_d)$. We know that 
$$ \tilde{A}\cdot \tilde{A}^{tr} = A\cdot A^{tr} + A_*\cdot A_*^{tr} = I_{d\times d}.$$ But the choice of the orthonormal basis of $span(A)$ and that of $span(A_*)$ are independent, therefore we may conclude that $A\cdot A^{tr}$ is independent of the basis for its $k$-space. In particular, $A\cdot A^{tr}= B\cdot B^{tr}$. We conclude that $S^1 = T^1$ and $J\star S^1 = J\star T^1$.
\end{proof}

The following theorem says that the $k$ dimensional martingale run on a $\bR^d$ Brownian motion $Z$ runs on a $\bR^d_k$ Brownian motion $S^1$. In other words, $Z$ splits essentially as ``$S^1 + S^2$" with $X$ running parallel to $S^1$ and orthogonal to $S^2$.

\begin{theorem}
Following the notation of Theorem \ref{mainT2}, a $k$-dimensional martingale $X=\int H\cdot dZ$ has a representation with respect to an $\bR^{d}_k$ Brownian motion $S^1$: 
\begin{equation}
X = \int \sqrt{k}H \cdot dS^1.
\end{equation}
\end{theorem} 
\begin{proof}
Let us think on the action of the matrix $H$. The rows of $H$ span a space for which $\vec{u}_1, \cdots, \vec{u}_k$ form an orthonormal basis. Therefore for any $v\in\bR^d$, 
$$H\cdot v = H\cdot \vec{u}_1 \vec{u}_1\cdot v + \cdots + H\cdot \vec{u}_k \vec{u}_k\cdot v.$$
Changing to the probabilistic setting, we have
\begin{eqnarray*}
 H\cdot dZ &=& \sum_{i=1}^k H\cdot \vec{u}_i \vec{u}_i\cdot dZ \\
&=& \sum_{i=1}^k H\cdot d\vec{U}^i \\
&=&  H \cdot d(\sum_{i=1}^k \vec{U}^i) \\
&=& \sqrt{k} H \cdot dS^1.
\end{eqnarray*}
\end{proof}

\begin{Problem}
If in Definition \ref{BM}, $W$ is measurable with respect to an $\bR^d$-Brownian motion $Z$, then there should be corresponding 
restrictions on the possibilities for $n$, $m$ and $K$. Find these restrictions.
\end{Problem}

\begin{Problem}
Find distribution and path properties for suitable classes of $\bR^{n\times m}_K$ processes; characterize in terms of infinitesimal generators.
\end{Problem}

\subsection{Regular and Exact $\bR^{n\times m}_K$ Brownian motions}
Consider the standard Brownian motion $Z= (Z_1, \cdots, Z_d)$. We can think of it as the cross-product process of $d$ one dimensional $\bR^1_1$ Brownian motions that are mutually orthogonal. Thus, it is an $\bR^{1\times d}_{{\bf{1}}_d}$ Brownian motion. On the other hand, $Z$ can also be thought of as the column sum
\[ Z= \begin{pmatrix} Z_1 \\ 0 \\ \vdots \\ 0\end{pmatrix}+ \begin{pmatrix} 0 \\ Z_2 \\ \vdots \\ 0\end{pmatrix}+ \cdots + \begin{pmatrix} 0 \\ 0 \\ \vdots \\ Z_d \end{pmatrix}\]
of $d$ one dimensional $\bR^d_1$ Brownian motions, that are mutually orthogonal. The coordinate processes in the sum form an ``orthonormal basis" for the space of $\bR^d$ martingales measurable with respect to $Z$. It is obtained by projecting $Z$ in the standard basis directions of $\bR^d$. Thus,
\[  \begin{pmatrix} Z_1 \\ 0 \\ \vdots \\ 0\end{pmatrix} =\int e_1^{tr}\cdot e_1\cdot dZ,\]
and so on. In the same way, $Z = \vec{U}^1 + \cdots + \vec{U}^d$, where $\vec{U}^i$ is the ``projection" of $Z$ on the direction $\vec{u}_i$. Thus, $Z$ may be thought as isomorphically associated with the $\bR^{d\times d}_{{\bf{1}}_d}$ Brownian motion $\vec{U}$. 

A rather remarkable property of $\vec{U}$ that distinguishes it among the general $\bR^{d\times d}_{{\bf{1}}_d}$ Brownian motions is that adding the coordinate processes adds the dimensions.  An $\bR^{2\times 2}_{{\bf{1}}_2}$ Brownian motion like ${V} = \left( \begin{pmatrix} Z_1 \\ 0 \end{pmatrix}, \begin{pmatrix} Z_2 \\ 0 \end{pmatrix}\right)$ clearly does not have this property. We first classify $\bR^{n\times m}_K$ Brownian motions that have this property.

\begin{definition}\label{regular_BM}
An $\bR^{n\times m}_K$ Brownian motion $\vec{W}$ is {\bf{regular}} if for any $1\leq i_1<\cdots< i_r\leq m$, we have
\begin{equation}\label{sum_dim}
Dim(\vec{W}^{i_1}+\cdots + \vec{W}^{i_r}) = k_{i_1}+\cdots + k_{i_r}.
\end{equation}
\end{definition}
The following proposition shows that $\vec{U}$ is a regular Brownian motion.

\begin{proposition}\label{special_property}
The process $\vec{U}$ has the following properties.
\begin{enumerate}
\item $Dim(\vec{U}^i) = 1$ for each $i$,
\item Let $d_0 = 0$, $d= k_1 + \cdots + k_m$ and $d_r = k_1+\cdots+k_r$. If $$\vec{W}^r = \vec{U}^{d_{r-1}+1} + \cdots + \vec{U}^{d_r},$$ then $Dim(\vec{W}^r) = k_r$, $\vec{W}^r$ and $\vec{W}^s$ are mutually orthogonal for $r\neq s$, and $\langle \vec{W}^r\rangle_t = k_r t$ for each $r$.
\end{enumerate}
\end{proposition}
\begin{proof}
For the first assertion, simply observe that the $d\times d$ integrand matrix of $\vec{U}^i$ is simply $\vec{u}_i^{tr}\cdot \vec{u}_i$, which clearly has rank $1$ since every column is a multiple of $\vec{u}_i^{tr}$ (or since $\vec{u}_i$ has rank $1$).

For the second part, the $d\times d$ matrix corresponding with $\vec{W}^r$ is 
$$\vec{u}_{d_{r-1}+1}^{tr}\cdot \vec{u}_{d_{r-1}+1} +  \cdots + \vec{u}_{d_r}^{tr}\cdot \vec{u}_{d_r}.$$
If $V_r$ is the matrix having rows $\vec{u}_{d_{r-1}+1}$ through $\vec{u}_{d_r}$, then the matrix for $\vec{W}^r$ is $(V_r)^{tr}V_r$. Since $rank(V_r)=k_r$ a.s., for a.e. $t$, and since $rank((V_r)^{tr}V_r) = rank(V_r)$, we have that the dimension of $\vec{W}^r$ is $k_r$. The span of the rows of $(V_r)^{tr}V_r$ is the span of the rows of $V_r$, hence it is clear that $\vec{W}^r$ and $\vec{W}^s$ are orthogonal for $r\neq s$. Finally, $$\langle \vec{W}^r\rangle_t =  \langle \vec{U}^{d_{r-1}+1}\rangle_t + \cdots + \langle \vec{U}^{d_r}\rangle_t = k_r t.$$ This is because of the orthogonality of the $\vec{U}^j$'s and since $\langle\vec{U}^j\rangle = t$, for all $j$.
\end{proof}

Most of our results in this section have been about $\vec{U}$. Since such a process may have more immediate applications, we formally define it separately. Let $$n_0=0, n_j = k_1 + \cdots+k_j, K= (k_1, \cdots, k_m).$$

\begin{definition}\label{exactBM}
Let $T=(T^1, \cdots, T^m)$ be an $\bR^{n\times m}_K$ Brownian motion measurable with respect to some $\bR^n$ Brownian motion $B$. Suppose there is a predictible orthonormal matrix process $P = (P_1, \cdots, P_n)$ such that for each $r$,
\begin{equation}\label{require}
T^r = \frac{1}{\sqrt{k_r}}\sum_{i=n_{r-1}}^{n_r}\int P_i\cdot P_i^{t} \cdot dB.
\end{equation}
Then we call $T$ an {\bf{exact}} $\bR^{n\times m}_K$ Brownian motion.
\end{definition}
Geometrically, what we have done is start with a Brownian motion $B$ and a predictible \emph{frame field} process $(P_1, \cdots, P_n)$, then project $B$ onto the subspaces spanned by collections $P_{n_{r-1}}, \cdots, P_{n_r}$, and finally normalize the projected processes. This obtains $T^r$. So if the coordinate processes of $T$ correspond to normalized projections of a Brownian motion onto orthogonal spaces, then $T$ is exact. As we saw with $\vec{U}$, the sum of coordinate processes of an exact Brownian motion has dimension equalling the sum of dimensions.

The definition however is merely following the construction of $\vec{U}$. We pose below the problem of understanding regular and exact Brownian motions more generally.

\begin{Problem}
Let $W=(\vec{W}_1, \cdots, \vec{W}_m)$ be a $\bR^{n\times m}_K$ Brownian motion with $\vec{W}_j = \int H^j\cdot d\vec{Z}^j$ being a stochastic representation with respect to some $\bR^{d_j}$ Brownian motion $\vec{Z}^j$. Find the minimal conditions that the matrix processes $H^j$ and the Brownian motions $\vec{Z}^j$ must satisfy in order that $W$ be either a regular or an exact $\bR^{n\times m}_K$ Brownian motion.
\end{Problem}

\subsection{Standard orthogonaliy}
In an earlier version of this paper, the author defined $\bR^{n\times m}_K$ Brownian motion $\vec{W}$ using a different (weaker) notion of orthogonality for $\bR^n$ martingales, that $\vec{W}^i$ and $\vec{W}^j$ are orthogonal if $\vec{W}^i\cdot\vec{W}^j$ is a martingale. This was based on the author's work in \cite{Ja1}, where we used such a concept (called ``standard orthogonality") for $\bC^n$ martingales. There are two reasons why it is tempting to assume importance for standard (or \emph{IP}, for inner-product) orthogonality. First, since we are dealing with $\bR^n$ processes, it is natural to consider their interaction on the basis of the outer dot product. This was our idea in \cite{Ja1}. 

The second reason is that if we take a martingale $\int K\cdot dW$, we see that the total behavior of the process is based on the dot products $K^j\cdot dW$. We would want therefore that $K^j$ and $dW$ be `orthogonal' in an appropriate dot-product sense, rather than requiring the more stringent condition that all coordinates of $dW$ be independent of $K$. This is related to the issue discussed in section \ref{Results} and to Question  \ref{indep_orth}. Based on this role of the dot product in stochastic representations, one can ask whether defining orthogonality for martingales in terms of the dot product may prove useful.

For this paper however, we decided that the well-known Definition \ref{orth_def} of orthogonality has clear applications with the process $\vec{U}$, hence is likely more important. The work in \cite{Ja1} should also be re-analyzed in terms of Definition \ref{orth_def}.

\subsection{Cross product and Dimension}
In the spirit of definitions for general Brownian motions, let us consider a cross-product process $Y = (Y_1, \cdots, Y_m)= $
$$ (\int {H}^1\cdot dZ, \cdots, \int {H}^m\cdot dZ)$$ where each ${H}^j$ is $n\times d$ and $Dim(Y_i) = k_i$. What should be the dimension of $Y$? Suppose we consider $Y$ as a joint $n\cdot m$ coordinate process, then $Y$ will have a total dimension if the matrix
$$ H = \begin{pmatrix} {H}^1\\ \vdots \\ {H}^m\end{pmatrix}$$
has a constant rank a.s., for a.e. $t$. This need not be the case in general. However, \emph{if the $Y_j$'s are mutually orthogonal}, then we can easily verify that the total dimension
$$ Dim(Y) = k_1 + \cdots + k_m.$$
Considered in this way, we see that the dimension of $\bR^{n\times m}_K$ Brownian motion is $k_1+\cdots+ k_m$.

The other way of looking at this is to consider the cross product as a concatenation of distinct processes. In this case, we should regard that $Y$ has a multi-index dimension 
\begin{equation}\label{cross_dim2}
Dim(Y) = K = (k_1, \cdots, k_m).
\end{equation}
Our definition of $\bR^{n\times m}_K$ Brownian motion has this simpler perspective; whether we should consider the total dimension can be decided in connection with particular applications.

\section{Martingale transforms}
\subsubsection{Left multiplication}
Given a predictible $m\times n$ matrix process $A$, it can act by left multiplication on a martingale $X= \int H\cdot dZ$ with $n\times d$ matrix $H$:
\begin{equation}\label{left_transform}
Y= A\star X = \int A\cdot H\cdot dZ.
\end{equation}
$A\star X$ is called a martingale transform of $X$ by matrix $A$. The mapping $$A: X\rightarrow Y$$ is similar to a mapping between two manifolds. If we write $X = X^1+X^2$ as the sum of two martingale sub-processes, then clearly each subprocess gets mapped to a subprocess of $Y$. There should be much that can be said regarding the interconnection between $A$, $H$ and Dimension. For instance, we can ask how such a transform maps between the ``manifolds" of processes having fixed dimensions.
\subsubsection{Right multiplication}
Applications suggest that we will also be interested in right multiplication transforms as well. If $B$ is a $d\times d$ matrix process, then we can define 
\begin{equation}\label{right_transform}
Q = X\star B = \int H\cdot B\cdot dZ.
\end{equation}
As an example, let $H = (h_1, h_2)\in \bC^2$ and let $B = \begin{pmatrix} 1 & i \\ i & -1\end{pmatrix}$. Then $$H\cdot B = (h_1+ih_2, i(h_1+ih_2)).$$ This right transformation by matrix $B$ is well known as giving the martingale associated with the Beurling-Ahlfors transform; see \cite{Ja1}. (It is however expressed in literature as left multiplication due to a difference in notation, where $H$ is treated as a column vector.)  Let us restrict attention to a space of non-stagnant martingales, i.e. having $H$ non-zero almost surely, for all $t$. Then all non-zero martingales have complex dimension $1$. Observe that 
$$ X = X^1+X^2 = \int \frac{(h^1 + i h^2)}{2}d(Z_1-iZ_2) + \int \frac{(h^1 - i h^2)}{2}d(Z_1+iZ_2),$$
a sum of two complex-orthogonal martingales. The left-kernel of $B$ is the vector subspace of \emph{conformal} martingales of the form $\int \alpha d(Z_1+iZ_2)$, hence we have $X^2\star B \equiv 0$. According to our definition for the dimension of martingale spaces in \cite{Ja1}, both the kernel and cokernel of $B$ have martingale space-dimension $=1$.

\begin{remark}
The question may be asked as to how the definition for dimension of martingales given in this paper will expand our understanding of martingale spaces. Perhaps some of the definitions and ideas will acquire clarity or require refinements.
\end{remark}

\section{Some thoughts and directions}

\subsection{``Manifolds" of martingales}
Let $\mathcal{F}_{n,k}(Z)$ be the collection of all $k$ dimensional $\bR^n$-valued processes that are measurable with respect to $Z$: $\mathcal{F}_{n,k}(Z)=$
\begin{equation}\label{Space-d}
  \{ \int A \cdot dZ: A \textrm{ is } n\times d, \textrm{ predictible}, rank(A) = k \textrm{ a.s., a.e. t} \}
\end{equation}
$\mathcal{F}_{n,k}(Z)$ can intuitively be considered as a sub-manifold within the vector space of all $Z$-measurable $\bR^n$-martingales.
Given any $n \times d$ predictible matrix process $H$ of rank $k$, we can think of $H$ as a function from $\mathcal{F}_{d, d}(Z)$ to  $\mathcal{F}_{n,k}(Z)$. By the theory established in the paper, we know that any $W\in \mathcal{F}_{d,d}$ splits into an orthogonal sum $W_H + W_H^\perp$, where $W_H\in \mathcal{F}_{d, k}$, $W_H^\perp\in \mathcal{F}_{d, d-k}$ and $$H\star W = H\star W_H, \hspace{4mm} H\star W_H^\perp \equiv 0.$$
One can ask about the nature of these ``sub-manifolds" and the behavior of $H$ between them.

\subsection{Martingale-valued mappings}
If we compose a smooth mapping $f:\bR^d\rightarrow\bR^n$ with $Z$, we obtain a semi-martingale 
\begin{eqnarray*}
 f(z_0 +Z_t)&=& f(z_0)+\int Df(z_0+ Z_s)\cdot dZ_s + Q_t(f,z_0) \\ &=& X_t(f, z_0)+Q_t(f, z_0) 
\end{eqnarray*}
where $Q$ is the bounded variation process. It is clear that if $f$ is a special map like a diffeomorphism, submersion, immersion, etc, then the martingale part $X_t(f, z_0)$ will have a certain fixed dimension. We can regard the semi-martingale as having the same dimension as its martingale part.

What makes this interesting is that we have a function $X_t(f,\cdot): \bR^d\rightarrow M\subset \mathcal{F}_{n,k}$, and in an appropriate sense, the function is continuous. Here as well, one can think of $M$ as a ``manifold" of martingales. Letting $H_s(z_0) = Df(z_0+Z_s)$, we find that for each $z_0$,
there is a decomposition $$ z_0 + Z_t = z_0 + Z(z_0,H_t) + Z(z_0, H_t)^\perp$$ such that
$$ H(z_0)\star Z = H(z_0)\star Z(z_0, H_t), \hspace{4mm} H(z_0)\star Z(z_0, H_t)^\perp \equiv 0.$$
$H$ annihilates not just $Z^\perp$ but an entire linear subspace of $\mathcal{F}_{d,d-k}$. Thus, if $V\subset \mathcal{F}_{n,k}$ is in the range of $H$, then 
$$H^{-1}(V) = \{(z, W): H(z)\star W \in V\}$$
is an interesting bundle of linear coset spaces.

Further, if $g:\bR^n\rightarrow\bR^m$, then $Dg$ induces the matrix process and transform on the martingale map $X(f)$ by
$$ Dg\star X_t(f) = \int_0^t [Dg(f)\cdot Df](z+Z_s)\cdot dZ_s.$$
For specially chosen $g$, we should be able to investigate the action of $Dg$ on $X(f)$.

\subsection{Relative to Stopping times} If we want to think that a martingale is analogous to a manifold, then naturally we want to understand associated martingales that may be compared to submanifolds of a manifold. (Clearly we cannot call them sub-martingales!) This paper already shows how if we use martingale transforms with appropriate matrix processes, then the base Brownian motion splits into orthogonal sum of sub-processes, which may be regarded as examples of martingale sub-processes. We can do this even when the base process is already a martingale, and not necessarily Brownian motion.

Another direction is to employ stopping times. A martingale $X$ with filtration $\mathcal{F}_t$ stopped at a stopping time $\tau$ is intuitively analogous to stopping on a sub-manifold of a manifold. Then the process $Y_t = X_{\tau+t}-X_\tau$ is a martingale with respect to the filtration $\mathcal{F}_{\tau+t}$. Looking at $Y_t$ is like looking at $X$ relative to the stopped random variable $X_\tau$; this is analogous to looking at the manifold from a submanifold. In fact, when we define dimension of a manifold, we go to each point on it and verify that locally the manifold is equivalent to Euclidean space. Likewise, we can ask whether for all stopping times $\tau$, the martingale $X_{\tau+t}-X_\tau$ with filtration $\mathcal{F}_{\tau+t}$ also has dimension $k$. Moreover, we can take such a ``local" understanding as the starting point and see if anything more subtle can be observed either with regard to the stopping times, or with the general definition itself and the consequences.

\subsection{Homotopy}
Observe that $\varphi: [0,1]\rightarrow \mathcal{F}_{2,1}\cup \mathcal{F}_{2,2}$ defined $$\varphi(t) = (\sqrt{t}Z_1, \sqrt{1-t}Z_2)$$ is a homotopy between two $1$ dimensional $\bR^2_1$ Brownian motions $(Z_1, 0)$ and $(0, Z_2)$. The homotopy's values in $(0,1)$ are however always $\bR^2_2$ Brownian motions, of dimension $2$.
More generally, one can regard a martingale of dimension $k$ as a suitable continuous limit of martingales of higher dimension; hence, the lower dimensional martingales can be seen as being on the boundary of the set of higher dimensional ones. 

\subsection{Relative to Frame fields} Let $(\vec{v}_1, \cdots, \vec{v}_d)$ be a Euclidean frame field in $\bR^d$, i.e. an orthonormal basis at each point. Split into two parts $(\vec{v}_1, \cdots, \vec{v}_k)$ and $(\vec{v}_{k+1}, \cdots, \vec{v}_d)$, spanning $V^1$ and $V^2$. Then any $X= \int J\cdot dZ$ splits as $X^1+X^2= \int J^1\cdot dZ+ \int J^2\cdot dZ$. The rows of $J$ are projected into $V^i$ to obtain the rows of $J^i$. 
\begin{Problem}
Choose process $X$ and frame fields (or space-fields $V^i$), and compare the associated processes $X$, $X^1$ and $X^2$.
\end{Problem}

\subsection{Dimension of Martingale vs Filtration}\label{AF_dim}
 The author found in \cite{Hi} and \cite{Hi2} definitions for dimension of a filtration $\mathcal{F}$ and an $AF$-dimension for martingales. Indeed, this appears a well developed and intricate subject; we just wish to point out that our approach and definitions seem to be different. $Dim(\mathcal{F})$ is the minimal number of $\mathcal{F}$-martingales $\{M_1, \cdots\}$ needed so that any $\mathcal{F}$-martingale $Y$ has a stochastic representation $\int A\cdot dM$. In particular, the Brownian filtration of $\bR^d$ Brownian motion has dimension $d$. Observe then that the filtration $\mathcal{F}^X$ generated by a martingale $X$ will automatically have a value, the dimension of $\mathcal{F}^X$, corresponding to $X$. We could have identified the same value as the Dimension of $X$. The $AF$-dimension of $X$ seems quite similar to this, although M. Hino \cite{Hi2} mentions that the equivalence of these definitions is not yet known.

Alternately, we can observe that there is a $minimal$ $d$ for which there exists an $\bR^d$ Brownian motion $Z$ and a possibly extended filtration $\mathcal{F}^Z \supset \mathcal{F}^X$ such that $X$ is measurable with respect to $(Z, \mathcal{F}^Z)$ and has a stochastic representation $\int H\cdot dZ$. We could have identified $Dim(X)$ with $Dim(Z) = d$. Still another alternative, we can start with this same representation $X= \int H\cdot dZ$. Then there is a minimal $k$ (possibly smaller than $d$) and an $\bR^k$ Brownian motion $(W, \mathcal{\tilde{F}}^X)$ such that $\mathcal{F}^X\subset \mathcal{\tilde{F}}^X \subset \mathcal{F}^Z$, and we have the representation $X=\int K\cdot dW$. We could consider this value of $k$ as a notion of dimension of $X$. 

This paper's definition however makes a stronger requirement. $X=\int K\cdot dW$ must also have $rank(K)=k$ a.s., a.e. $t$ in order to have Dimension $k$. (This requirement implies $\mathcal{\tilde{F}}^X = \mathcal{F}^X$.) Our reason for this added requirement is the geometric intuition described earlier, that seems similar to topology and geometry. Studying special classes of martingales based on this notion of dimension may give similar theories for stochastic processes. 
(Note also that we have the notion of Brownian Dimension from Definition \ref{Brownian_D}, which is based on an even stronger requirement.)

\begin{remark}
Observe here as well, there can be useful variants.  For instance, we can start with $(X, \mathcal{F}^X)$. Suppose $Dim(\mathcal{F}^X) = k$, so that there are $\mathcal{F}^X$-martingales $M = (M^1, \cdots, M^k)$ such that $X = \int N\cdot dM$. If $rank(N)=r$ a.s., a.e. $t$, then we can say that the dimension of $X$ is $r$. It appears then, we will necessarily have $r=k$, with the $M^j$ being suitable coordinates of $X$. 
\end{remark}

\begin{Problem}
Classify martingales based on these different possible definitions for Dimension. Find the precise relations.
\end{Problem}

\begin{remark} The ideas motivating this paper are primarily from manifold theory and the work of section 3 in \cite{Ja1}. The author saw \cite{Hi} and \cite{Hi2} only after or in the process of writing the first drafts of this paper. However, it did motivate deeper thought on filtration and a realization of the limitation in the implication of Theorem \ref{mainT2}. Thankfully, a broader perspective was found in \cite{KS} and Theorem \ref{KS_theorem}. There is also another analogous notion of intrinsic dimension for signal processing that is similar; see wikipedia for a basic idea. 
\end{remark}

\subsection{An error correction for \cite{Ja1}}
The following misstatement in a definition occurs in \cite{Ja1}. We correct it here while noting that the correction in no way whatsoever changes the main results of the paper. In the \cite{Ja1}, we had introduced three martingale spaces (see for instance in section 5):
\begin{eqnarray*}
\mathcal{M}^0 &=& \left\{\int_0^t \nabla U_\varphi(B_s)\cdot dZ_s: U_\varphi \textrm{ heat-ext of } \varphi\in L^2(\bC)\right\} \\
\mathcal{M}^1 &=& I\star\mathcal{M}^0 + J\star\mathcal{M}^0, \textrm{   where } I = \begin{pmatrix}1 & \\ & 1\end{pmatrix}, J = \begin{pmatrix} & -1 \\ 1 & \end{pmatrix} \\
\mathcal{M} &=& \left\{A\star\varphi = \int_0^t A\nabla U_\varphi(B_s)\cdot dZ_s: \varphi\in L^2, A \textrm{ any } 2\times 2 \textrm{ matrix}\right\} 
\end{eqnarray*}
The last space $\mathcal{M}$ is referred to at many places as the space of martingale transforms of $\mathcal{M}^0$, and it is also claimed that $\mathcal{M}^1\subset \mathcal{M}$.  However it is clear that $\mathcal{M}$ is not a vector space as stated since it need not include martingales like $A\star\varphi + B\star\psi$, i.e. not closed under addition. The following is the necessary definition that makes the paper precise. Define 
\begin{equation}\label{correction}
\mathcal{M} = I_{11}\star\mathcal{M}^0 + I_{12}\star\mathcal{M}^0 +I_{12}\star\mathcal{M}^0+I_{12}\star\mathcal{M}^0,
\end{equation}
where $I_{ij}$ has $1$ in the $(i,j)$ slot of the matrix and $0$ otherwise. Thus $\mathcal{M}$ is the vector space of martingales generated by all martingale transforms of $\mathcal{M}^0$.

\subsection{Acknowledgment}
The author thanks Professor Rodrigo Ba\~nuelos for reading the paper and for his guidance.

\markboth{}{\sc \hfill \underline{References}\qquad}

\end{document}